\documentclass[12 pt]{article}
\usepackage{amsmath}
\usepackage{amsfonts}
\usepackage{amssymb}
\usepackage{amsthm}
\usepackage{thmtools}
\usepackage{graphicx}
\usepackage{hyperref}
\usepackage{mathtools}
\usepackage{enumitem}
\usepackage{subcaption}
\usepackage{xcolor}
\usepackage{geometry}
\usepackage[backend=biber, style=alphabetic, sorting = nyt]{biblatex}
\usepackage{biblatex}
\renewbibmacro{in:}{}
\usepackage{multicol}
\usepackage{float}

\newtheorem{thm}{Theorem}[section]
\newtheorem{lem}[thm]{Lemma}

\newtheorem{cor}[thm]{Corollary}
\newtheorem{rmk}[thm]{Remark}
\newtheorem{conj}[thm]{Conjecture}

\newtheorem{definition}{Definition}
\allowdisplaybreaks

\bibliography{dehnsurgery.bib}

\title{Positive 2-bridge knots and chirally cosmetic surgeries}
\author{Michael Huang \thanks{University of California, Berkeley, Berkeley, CA. ({\tt michael\_huang@berkeley.edu})},
Zelong Li \thanks{University of California, Los Angeles, Los Angeles, CA. ({\tt lizelong831@ucla.edu})},
Rahi Tanaz \thanks{University of California, Los Angeles, Los Angeles, CA. ({\tt rtanaz@ucla.edu})},
Chengyi Zhang \thanks{University of California, Los Angeles, Los Angeles, CA. ({\tt chriszhang09@ucla.edu})}}
\date{}

\begin{document}
\graphicspath{ {figures/} }

\maketitle

\begin{abstract}
    In this paper we verify that with the exception of the $(2, 2n+1)$ torus knots, positive 2-bridge knots up to 31 crossings do not admit chirally cosmetic surgeries. A knot \(K\) admits chirally cosmetic surgeries if there exist surgeries \(S^3_r\) and \(S^3_{r'}\) with distinct slopes \(r\) and \(r'\) such that \(S^3_r(K) \cong -S^3_{r'}(K)\), where the negative represents an orientation reversal. To verify this, we use the obstruction formula from \cite{konstantinos} which relates classical knot invariants to the existence of chirally cosmetic surgeries. To check the formula, we develop a Python program that computes the classical knot invariants \(a_2\), \(a_4\), \(v_3\), \(\det\), and \(g\) of a positive 2-bridge knot. 
\end{abstract}


\section{Introduction}

\subsection{Motivation}

Given a knot $K\in S^3$ and a slope $r\in \mathbb{Q}\cup\{\infty\}$, the manifold resulting from Dehn surgery with slope \(r\) on $K$ is denoted by $S_r^3(K)$. If there exist distinct slopes $r$ and $r'$ such that $S_r^3(K) \cong S_{r'}^3(K)$, then the surgeries \(S^3_r\) and \(S^3_{r'}\) are said to be purely cosmetic. If instead $S_r^3(K) \cong -S_{r'}^3(K)$, then the surgeries are said to be chirally cosmetic. 

Originally proposed by Bleiler in \cite{kirby} (Problem 1.81 A), the purely cosmetic surgery conjecture claims that there exist no purely cosmetic surgeries for nontrivial knots in $S^3$. Although this conjecture is still open, there have been many results on special families of knots. It was proved by Ichihara, Jong, Mattman, and Saito that 2-bridge knots admit no purely cosmetic surgeries \cite{Ichihara_2021}. Furthermore, Hanselman has proved the purely cosmetic surgery conjecture for all knots for which each prime summand has 16 or fewer crossings \cite{hanselman}. Varvarezos has also shown that 3-braid knots do not admit purely cosmetic surgeries \cite{konstantinos_braid}. 

Using machinery from modern Heegaard Floer homology, Varvarezos was also able to construct obstructions to chirally cosmetic surgeries for special families of knots. We use the results from Varvarezos in \cite{konstantinos} to prove the non-existence of chirally cosmetic surgeries on positive 2-bridge knots, with the exception of $(2,2n+1)$ torus knots.

\subsection{Main result}

We aim to prove the following result. 

\begin{thm}
    \label{main}
    Positive 2-bridge knots up to 31 crossings do not admit chirally cosmetic surgeries, with the exception of $(2,2n+1)$ torus knots. 
\end{thm}
By showing that this result holds for positive 2-bridge knots up to 31 crossings, we will have verified that the conjecture holds for over 1.3 million knots. To prove Theorem \ref{main}, we will first introduce the background necessary to understand a useful obstruction formula and prove some important theorems and lemmas along the way. We will then describe the Python program we developed in order to calculate the knot invariants used in the obstruction formula, and finally conclude with our results and further directions.

\subsection{Notation}

Let $K$ be a knot. We will use the following notation throughout the paper. 
\begin{align*}
    &\text{Jones polynomial: } V_K(t) \\
    &\text{Alexander polynomial: } \Delta_K(t) \\
    &\text{Conway polynomial: } \nabla_K(z)\\
    &\text{Conway polynomial coefficients: } \nabla_K(z) = a_0 + a_2z^2+a_4z^4 + ... \\
    &(p,q) \text{ Rational knot: } C(p,q) \\
    &\text{Continued fraction: } p/q = a_1+\frac{1}{a_2+\frac{1}{...+\frac{1}{a_n}}} =: [a_1,a_2,...,a_n] \\
    &\text{Bridge Number: } Br(K)
\end{align*}

\section{Background}

\subsection{Positive 2-bridge knots}
In this section, we define the class of positive 2-bridge knots and give a result about continued fraction characterization. 

\begin{definition}
    The {\normalfont sign} associated to a crossing of the knot K is indicated by Figure \ref{sign} below. 
    \begin{figure}[h]
        \begin{center}
            \includegraphics[scale=1]{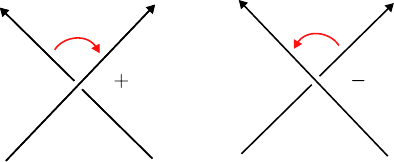}
        \end{center}
        \caption{Signed crossings}
        \label{sign}
    \end{figure}
    
\end{definition}
A crossing is positive if the understrand is rotated clockwise to match the direction of the overstrand, while a crossing is negative if the understand is rotated counter-clockwise to match the direction of the overstrand.

\begin{definition}
    A knot $K$ is {\normalfont positive} if all of its crossings are positive in some projection of $K$.
\end{definition}

\begin{definition}
    The {\normalfont bridge number} $Br(K)$ of a knot $K$ is defined to be the minimal number of bridges required in any projection of a $K$. A {\normalfont bridge} in \(K\) is an arc which has at least one overcrossing.
\end{definition}

\begin{rmk}
    $Br(K)$ is a knot invariant.
\end{rmk}

\begin{definition}
    A {\normalfont 2-bridge knot} is a knot with $Br(K) = 2$. A {\normalfont positive 2-bridge knot} is a 2-bridge knot which is also a positive knot.
\end{definition}

As currently defined, the class of positive 2-bridge knots is difficult for a computer program to enumerate, so we seek a more tractable definition. 

\begin{definition}
    A {\normalfont tangle} in a disk in the plane of projection of a knot $K$ such that $K$ crosses the boundary circle four times. A {\normalfont rational tangle} is a tangle that may be unwound into one of the two elementary 2-tangles in Figure \ref{elementary} by repeatedly twisting the endpoints. 
\end{definition}
\begin{figure}[h]
    \begin{center}
        \includegraphics[scale=0.9]{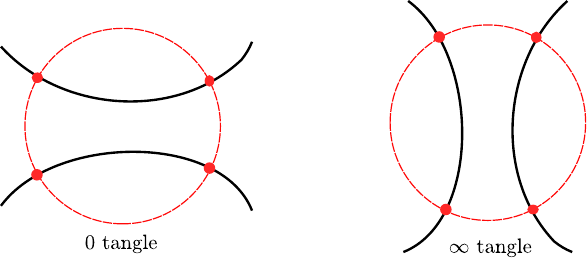}
    \end{center}
    \caption{Elementary 2-tangles}
    \label{elementary}
\end{figure}

\begin{definition}
    A {\normalfont rational link} is the link that results from closing off the ends of a rational tangle by plait closure. If the resulting link is a knot, it is known as a {\normalfont rational knot}.  
\end{definition} 

\begin{thm}[\S 3.2, \cite{adams}]
    The set of 2-bridge knots is identical to the set of rational knots. 
    \label{rational}
\end{thm}

Corresponding to the number of twists required before a rotation needed to completely unwind a rational tangle into one of the elementary 2-tangles, we assign a continued fraction decomposition \([a_1, \dots, a_n]\) to a rational knot. We also assign it the rational number \(p/q\) that its continued fraction sums to, and hence denote a rational \(p/q\) knot \(C(p, q)\). Thus, from Theorem \ref{rational}, we may also characterize 2-bridge knots by a rational number \(p/q\) and a continued fraction decomposition \([a_1, \dots, a_n]\). 

\begin{thm}[Thm. 3.1, \cite{tanaka}] 
    If a 2-bridge knot $C(p,q), \frac{p}{q} = [a_1,...,a_{2n+1}]$ satisfies the following conditions:
    \begin{enumerate}
        \item[i)] Each $a_i$ is positive
        \item[ii)] $a_i$ is even if the index is even 
    \end{enumerate}
    then $C(p,q)$ is a positive knot. \footnote{For 2-bridge knots, the strongly quasipositive condition is equivalent to positivity, see \cite{ozbagci} remark 1.5}
    \label{3.1}
\end{thm}

The converse of Theorem \ref{3.1} is given in Lemma \ref{2.3} below. With these results we obtain an alternative characterization of 2-bridge knots that is much better suited for a computer program. 

\begin{lem} [Cor. 2.2, \cite{ozbagci}]
    \label{cond_for_pos} 
    If $C(p, q)$ is a positive rational knot, then for some \(n \in \mathbb{Z}\) it has a continued fraction decomposition of the form
    \begin{align*}
        \frac{p}{q} = a_1 + \frac{1}{a_2+\frac{1}{\ddots + \frac{1}{a_{2n+1}}}}
    \end{align*}
    where each $a_{2i}$ is even for $i \in \{1,2,...,n\}$, and each $a_i > 0$.
    \label{2.3}
\end{lem}

\begin{rmk}
     We note that in the paper \cite{ozbagci}, the author uses the mirror image of the $C(p,q)$ knot we use. Therefore, their definition of $C(p,p-q)$ is equivalent to our definition of $C(p, q-p)$. However, due to the results of \cite{kauffman_rational}, we know that $C(p,q-p) = C(p, q)$, as $q \equiv q-p \mod{p}$. 
\end{rmk}

The schematic diagram for a positive 2-bridge knot corresponding to the continued fraction decomposition given in Lemma \ref{cond_for_pos} is illustrated in Figure \ref{normal_twist_notation} below. The number \(a_i\) inside each box denotes the number of half twists to be inserted. 
Note that the figure results in a knot when \(\sum_{i = 1}^{2n + 1} a_i\) is odd, while the figure results in a link when \(\sum_{i = 1}^{2n + 1} a_i\) is even. 

\begin{figure}[h]
    \begin{center}
        \includegraphics[scale=1.3]{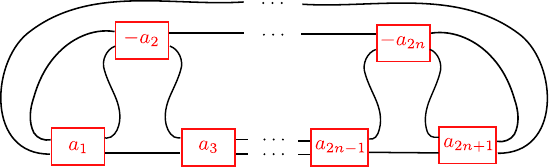}
    \end{center}
    \caption{Positive 2-Bridge knot}
    \label{normal_twist_notation}
\end{figure}

\begin{rmk}
    The negative sign on the upper row is due to how half-twists are signed in our notation. Some authors use the opposite notation where instead the bottom row would be negated. 
\end{rmk}

\subsection{(Chirally) cosmetic surgery}

In this section we give a brief overview of Dehn surgery and the cosmetic surgery conjecture. All surgeries unless otherwise specified will be done on the 3-sphere $S^3$. 

\begin{definition}
    Given a knot $K$ and a slope $r = \frac{p}{q} \in \mathbb{Q} \cup \{\infty\}$, {\normalfont Dehn surgery with slope \(r\)} (often abbreviated {\normalfont \(r\)-surgery}) on $K$ is the operation of removing a tubular neighborhood of $K \subset S^3$ and gluing in a solid torus along the boundary homeomorphism $\phi: T^2 \rightarrow T^2$, $m \mapsto p \mu + q \lambda$. The resulting manifold is denoted \(S^3_r(K)\). 
    \label{dehn}
\end{definition}
Here, $m$ denotes the meridian on the boundary of the solid torus, $\mu$ denotes the meridian of the tubular neighborhood of $K$, and $\lambda$ denotes the unique Seifert longitude of $K$. It is well known that all the self homeomorphisms of $T^2$ up to isotopy are given in the form specified in Definition \ref{dehn}. 

\begin{rmk}
    By our convention, $\infty$-surgery is $\frac{1}{0}$-surgery. This is the identity since the meridian of the solid torus is glued back to the meridian of the tubular neighborhood, resulting back in our original knot \(K \subset S^3\).  
\end{rmk}

We illustrate the idea of of Dehn surgery with slope \(r = p/q\) on the trefoil in the Figures \ref{tubular neighborhood} and \ref{gluing} below. For more about Dehn surgery, see \cite{lickorish}, \cite{saveliev}, \cite{rolfsen}. 
\begin{figure}[h]
    \begin{center}
        \includegraphics[scale=1]{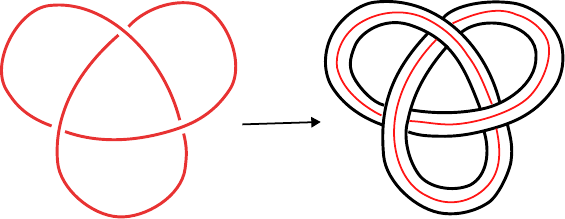}
    \end{center}
    \caption{Taking a tubular neighborhood of trefoil.}
    \label{tubular neighborhood}
\end{figure}

\begin{figure}[h]
    \begin{center}
        \includegraphics[scale=1]{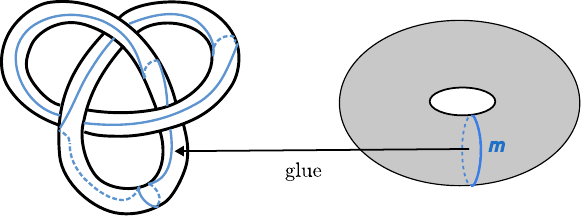}
    \end{center}
    \caption{Gluing meridian of solid torus to (p,q)-curve on torus boundary. Here, $p=q=1$.}
    \label{gluing}
\end{figure}

\begin{definition}
    Given a knot $K$, we say that $K$ admits {\normalfont purely cosmetic surgeries} if there exist two distinct slopes $r$ and $r'$ such that $S^3_r(K) \cong S^3_{r'}(K)$ up to orientation preserving homeomorphism. 
\end{definition}

\begin{conj}[Purely cosmetic surgery conjecture]
    No nontrivial knot in $S^3$ admits purely cosmetic surgeries. 
\end{conj}

We consider a related case of this conjecture which we will call the chirally cosmetic surgery conjecture. 

\begin{definition}
    A knot $K$ admits {\normalfont chirally cosmetic surgeries} if there exist two distinct slopes $r$ and $r'$ such that the manifolds resulting from the $r$ and $r'$ surgeries are homeomorphic after an orientation reverse, that is, $S^3_r(K) \cong -S^3_{r'}(K)$. 
\end{definition}

\begin{conj} [Chirally cosmetic surgery conjecture]
\label{chirallyconj}
    A knot $K$ does not admit chirally cosmetic surgeries if it is not the unknot, a $(2,2n+1)$ torus knot, or an amphichiral knot. 
\end{conj}

For an amphichiral knot $K$, $S^3_r(K) \cong -S^3_{-r}(K)$. Meanwhile, for $(2,2n+1)$ torus knots, there exist chirally cosmetic surgeries along the slopes $\displaystyle{\tfrac{2n^2(2m+1)}{n(2m+1)\pm 1}}$ for each positive integer $m$ (see \cite{ichihara_chirally}). 

\subsection{Knot invariants}
\label{invariants}

In this section we introduce relevant knot invariants for the obstruction formula from \cite{konstantinos}. We will also provide computationally scalable procedures for retrieving these invariants for positive 2-bridge knots.
\begin{definition}
    The {\normalfont genus} ${\normalfont g(K)}$ of a knot $K$ is the least genus of any Seifert surface for $K$. A {\normalfont Seifert surface} for $K$ is an orientable surface with one boundary component that is the knot $K$. 
\end{definition}

\begin{lem}
    \label{genus_from_seifert}
    Let $K$ be an alternating knot and $\Sigma$ be the Seifert surface obtained from a reduced alternating diagram of $K$ via the Seifert algorithm. Let $s$ be the number of Seifert circles in $\Sigma$ and $c$ be the number of crossings in the alternating diagram. Then, $K$ has genus 
    \[
        g(K) = \frac{1+c-s}{2}.
    \] 
\end{lem} 
\begin{proof}
    This result is proved using the Euler characteristic $\chi$ of $\Sigma$. Notice that any Seifert surface of a knot $K$ can be deformation retracted to be a graph with a vertex at each Seifert circle and an edge at each crossing. The Euler characteristic of $\Sigma$ is then given by the number of vertices minus the number of edges plus $1$, since there are two faces but we have one boundary component with boundary $K$. By definition, $\displaystyle{g = \tfrac{2-\chi}{2}}$ and we have $\chi = s-c+1$, thus we get $\displaystyle{g(K) \leq \tfrac{1+c-s}{2}}$. It is well-known that the Seifert algorithm when applied to a reduced alternating projection of a link yields a minimal genus Seifert surface (see \cite{gabai}), so we are done. 
\end{proof}

With Lemma \ref{genus_from_seifert}, we have a closed-form description of the genus of a positive 2-bridge knot. 

\begin{lem} 
\label{genus}
    If $K=C(p, q)$ has the continued fraction decomposition $[a_1,...,a_{2n+1}]$ satisfying Lemma \ref{cond_for_pos}. above, then \(K\) has genus
    \begin{align*}  
        g(K) = \frac{1}{2}\left(\sum\limits_{i=0}^{n} a_{2i+1} -1 \right)
    \end{align*}
\end{lem}
\begin{proof}
    Consider the standard projection of a positive 2-bridge knot as in Figure \ref{normal_twist_notation}, we give the projection an orientation as in Figure \ref{oriented knot}. We will apply the Seifert algorithm to obtain the Seifert surface of $K$ (note that all positive 2-bridge knot diagrams are alternating). First, we will take the oriented resolution (see Figure \ref{oriented_res}) at each crossing to obtain Figure \ref{seifert circles}. Next, we add half-twisted bands in the place of where each crossing used to be, obtaining Figure \ref{seifert surface}. Now, we can count the number of crossings and Seifert circles to get the genus from Lemma \ref{genus_from_seifert}. Clearly, $c = \sum_{i=1}^{2n+1} a_i$. For $s$, we can notice that each upper strand (corresponding to the $a_{2i}$) contributes $a_{2i}-1$ circles, and each lower strand $a_{2i+1}$ contributes only one circle. We group the circle from each lower strand to the next upper strand, with the exception of the last lower strand. This gives $s = \sum_{i=1}^{n} a_{2i} + 1 + 1$ where we have an extra circle from the big circle in the background. Plugging these values in, we get the desired formula.
\end{proof}

\begin{figure}[h]
    \begin{center}
        \includegraphics[scale=0.7]{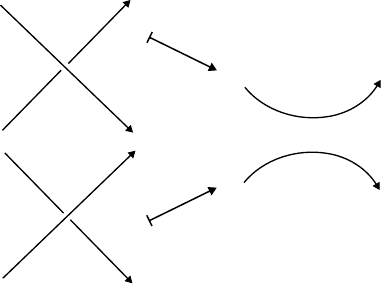}
    \end{center}
    \caption{Oriented resolution of a crossing}
    \label{oriented_res}
\end{figure}
The process of using the Seifert algorithm on a general 2-bridge knot to create a Seifert surface is illustrated in Figures \ref{oriented knot}, \ref{seifert circles}, and \ref{seifert surface} below. 

% fixed
\begin{figure}[h]
    \begin{center}
        \includegraphics[scale=1.2]{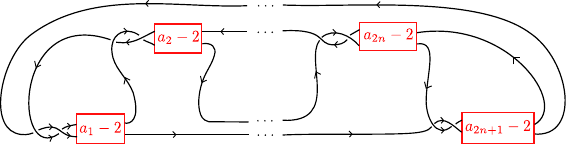}
    \end{center}
    \caption{Oriented knot K before crossing resolutions}
    \label{oriented knot}
\end{figure}
\begin{figure}[h]
    \begin{center}
        \includegraphics[scale=0.9]{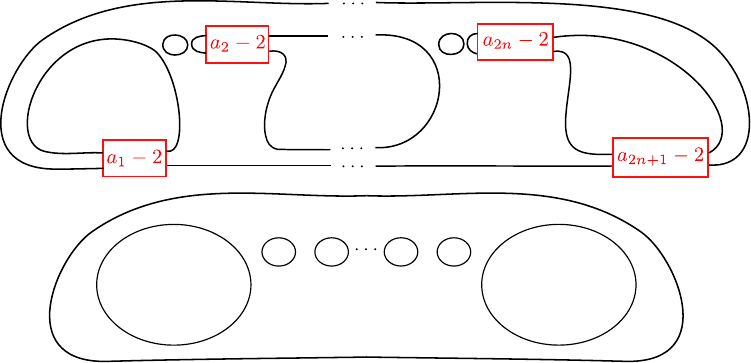}
    \end{center}
    \caption{Resolving crossings and resulting Seifert circles}
    \label{seifert circles}
\end{figure}

\begin{figure}[h]
    \begin{center}
        \includegraphics[scale=1]{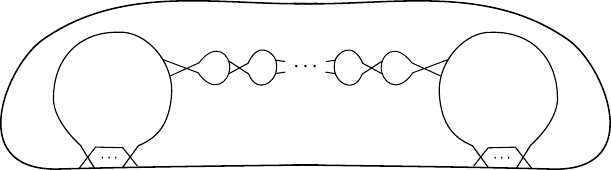}
    \end{center}
    \caption{Final Seifert surface}
    \label{seifert surface}
\end{figure}

\begin{lem}[Section 4, \cite{ichiharawu}]
    \label{lemma 1.7} 
    If $K = C(p, q)$ is a positive two bridge knot with $|p| > |q|$, then there exist $p', q' \in \mathbb{Z}$ such that $C(p',q') = C(p,q)$ and $\frac{p'}{q'}$ has an even continued fraction representation of the form
    \[
        \frac{p'}{q'} = a_1 + \frac{1}{a_2+\frac{1}{\ddots + \frac{1}{a_{2n}}}}  
    \]
    where $n \in \mathbb{N}$ and each of $a_i \neq 0$ is even for $i \in \mathbb{N}$. $p', q'$ can be obtained with the following conditions:
    \begin{enumerate}
        \item $(p', q') = (p, q)$ if $q$ is even
        \item if $q$ is odd, $p' = p$ and $q' = q \pm p$ with the restriction that $|q'| < |p'|$.  
    \end{enumerate}
\end{lem}

\begin{proof}
    To prove this lemma, we assume that we already have the construction that $|p| > |q|$ with $q$ even. We apply induction on $|q|$. 
    
    Base case: if $|q| = 2$, then we have $q = \pm 2$, then because we have $|p| > |q|$ and $p$, $q$ coprime, we must have $\frac{p}{q} = a_1 \pm \frac{1}{2}$, with $a_1 = \lfloor(p/q)\rfloor$ if $\lfloor(p/q)\rfloor$ even, and $a_1 = \lfloor(p/q)\rfloor+1$ otherwise. Since we have $|p| > |q|$, we cannot have $a_1 = 0$, which satisfies the conditions of a even continued fraction defined in the theorem.

    For the inductive step, we suppose that the even continued fraction exists for all such $p, q'$ with $|q'| \le 2n$. Let's consider $p,q$ with $q = 2(n+1)$. Then we rewrite $p = q + h$, where $h \in \mathbb{Z}$. Now we set $a_1$ as the even number in $\lfloor(p/q)\rfloor$ and $\lfloor(p/q)\rfloor+1$, then we can rewrite
    \[
        \frac{q+h}{q} =a_1 - (a_1 - \frac{q+h}{q})
    \]
    \[ 
        \frac{q+h}{q} = a_1 - \frac{1}{\frac{q}{(a_1 - 1)q-h}}
    \]
    It is crucial to realize that due to the choice of $a_1$, we have ensured that $|q| > |(a_1 - 1)q-h|$. Denote $q_1 = (a_1 - 1)q-h$. Also, because $q$ is even and $h$ is odd, we must have $q_1$ odd. Then we apply a similar procedure on $\frac{q}{(a_1 - 1)q-h}$. We have
    \[
    \frac{q+h}{q} = a_1 - \frac{1}{\frac{q}{q_1}} = a_1 - \frac{1}{a_2 - \frac{a_2q_1 - q}{q_1}} = a_1 - \frac{1}{a_2 - \frac{1}{\frac{q_1}{a_2q_1 - q}}}
    \]
    At this stage, if $a_2q_1 - q = 0$, then we are done. If otherwise, once again, we must have $|q_1| > |a_2q_1 - q|$, and also because $q_1$ odd, $a_2$ even, and q even, we have $a_2q_1 - q$ even. Then remember that $|q| > |q_1| > |a_2q_1 - q|$. Then by the inductive hypothesis, $\frac{q_1}{a_2q_1 - q}$ has an even continued fraction. Thus $\frac{p}{q}$ has an even continued fraction. 
\end{proof}

With this even continued fraction decomposition, we have another alternative way to obtain the genus of a positive 2-bridge knot.
\begin{thm}[\S 4, \cite{ichiharawu}]
    If K = C(p,q) and $K$ has the even continued fraction representation $[2b_1, 2c_1, \dots, 2b_n, 2c_n]$ as constructed in Lemma \ref{lemma 1.7}, then $g(K) = n$, where $n$ is half the length of the even continued fraction representation. 
\end{thm}

\begin{definition}
    The {\normalfont Jones polynomial} of any link $L$ can be characterized by the following skein relation:
    \[
        t^{-1}V(L_+) - tV(L_-) + (t^{-1/2} - t^{1/2})V(L_0) = 0
    \]
    where $L_-, L_+$, and $L_0$ represents the three ways of resolving a particular crossing for $L$ (see Figure \ref{skein}). We normalize $V(U) = 1$, where $U$ is the unknot.
\end{definition}
\begin{figure}[h]
    \begin{center}
        \includegraphics[scale=0.7]{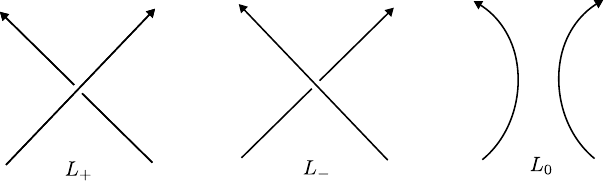}
    \end{center}
    \caption{\(L_+\), \(L_-\), and \(L_0\)}
    \label{skein}
\end{figure}

\begin{thm} [Proposition 4.4, \cite{ichiharawu}]
    If $K = C(p, q)$ and $K$ has the even continued fraction representation $[2b_1, 2c_1, \dots, 2b_n, 2c_n]$ as constructed in Lemma \ref{lemma 1.7}, then we can calculate knot invariant $v_3$ as below.
    \[
        v_3(K) = \frac{1}{2}\left(\sum\limits_{k = 1}^{n} c_k (\sum\limits_{i = 1}^k b_i)^2 - \sum\limits_{k = 1}^{n} b_k (\sum\limits_{i = k}^n c_i)^2 \right) 
    \]
    Alternatively, we can calculate \(v_3\) directly from the Jones polynomial with the following formula.
    \[
        v_3(K) = -\frac{1}{36}V_K'''(1) - \frac{1}{12}V_K''(1).
    \]
\end{thm}
In Section \ref{section 3}, we explain that our program always verifies that \(v_3 \neq 0\) before checking the obstruction formula given in Corollary \ref{obstruction_formula}. 

\begin{definition}
    \label{conway_skein}
    The {\normalfont Alexander polynomial} of any link $L$ can be characterized by the following skein relation: 

    \[ 
        \Delta(L_+) - \Delta(L_-) = (t^{1/2} - t^{-1/2})\Delta(L_0)
    \]
    with $\Delta(U) = 1$, where $U$ is the unknot. One can then find the {\normalfont Conway polynomial} of any link L by making the substitution $\Delta_L(t^2) = \nabla_L(t - t^{-1})$. The skein relation follows accordingly:
    \[
        \nabla(L_+) - \nabla(L_-) = z\nabla(L_0).
    \]
    See Figure \ref{skein} for \(L_+\), \(L_-\), and \(L_0\).
\end{definition}

For an alternative way to calculate the Alexander polynomial, see Appendix \ref{seifert_form}. Due to the higher time complexity of calculating determinants of large matrices, the recursive relation was chosen to be implemented in the code. 

From the Conway polynomial, we can extract a few more important invariants. 
\begin{definition}
    The $a_2$ and $a_4$ invariants of a knot $K$ are the coefficients of the $z^2$ and $z^4$ terms respectively in $\nabla_K(z)$. 
\end{definition}

\begin{definition}
    The {\normalfont determinant} of a knot is given by 
    \[
    \det(K) = |\Delta_K(-1)| = |\nabla_K(2i)|
    \]
    Equivalently, we can get the determinant of $K$ as the absolute value of substituting $z^2 = -4$ in $\nabla_K(z)$. 
\end{definition}

\subsection{Obstruction formula}
In this section, we will derive an obstruction formula to the existence of chirally cosmetic surgeries which follows the work Vavarezos in \cite{konstantinos}. Although the actual content of some of these formulas requires background in Heegaard Floer homology, we can state the main results without issue. 

\begin{thm}[2.8, \cite{konstantinos}]
\label{v_3 decomp}
    Let $K$ be knot satisfying $\tau(K) = g(K)$, and $v_3(K)\neq0$. If $S_{p/q}^3(K)\cong -S_{p/q'}^3(K)$ for some $p, q > 0, q'<0$ then
    \[
        v_3(K)(Vert(K) + 2g(K) - 1) = 7a_2(K)^2 - a_2(K) - 10a_4(K)
    \]
\end{thm} 
$Vert(K)$ and $\tau(K)$ are modern knot invariants derived from Heegaard Floer homology, but as they are not used in the final form of the obstruction formula, we will not discuss them further. See \cite{heegaard_floer} and \cite{manolescu} for precise definitions.  

\begin{rmk}
 Notice that in Theorem \ref{v_3 decomp} we are only considering the cases of $p,q >0,$ $q' < 0$. This is because for any non-$L$-space\footnote{An $L$-space knot is a knot $K$ such that $S^3_r(K)$ is an $L$-space for some $r\in \mathbb{Q}_+$. For more details, see \cite{ozvath}.}  knot K, if $S_r^3(K) = \pm S_{r'}^3(K)$, then $r, r'$ must have different signs \cite{konstantinos}. 
\end{rmk}

\begin{lem}[3.5, \cite{konstantinos}]
\label{obstruction for alternating knot}
    If K is an alternating knot with $\tau(K) > 0$, then
    \[
        Vert(K) = \frac{1}{2}\det(K) + \tau(K) - \frac{3}{2}
    \]
\end{lem}

We now have the necessary tools to modify the obstruction formula given in Theorem \ref{v_3 decomp} to only include the classical knot invariants discussed in Section \ref{invariants}. The resulting formula given in Corollary \ref{obstruction_formula} below is the final version of the obstruction formula we will use in our program. 
\begin{cor}
    \label{obstruction_formula}
    Let $K$ be a a positive 2-bridge knot which is also not a $(2,2n+1)$-torus knot. If $K$ admits a chirally cosmetic surgery, then
    \[
        v_3\left(\frac{1}{2} \det(K) + 3g - \frac{5}{2}\right) = 7a_2^2 - a_2 - 10a_4   
    \]
\end{cor}

\begin{proof}
    We first show that positive 2-bridge knots satisfy the conditions necessary for Theorem \ref{v_3 decomp} and Lemma \ref{obstruction for alternating knot}. From 1.8 in \cite{hedden}, we see that positive knots satisfy $\tau(K) = g(K)$. For knots with a positive diagram with $c$ crossings, we have $v_3(K) > c$ from \cite{stoimenow}. Since the unknot does not have two bridges, we have that $v_3(K) \neq 0$ for any positive two bridge knots. Next, from \cite{l_space}, we know that the only rational knots which are also L-space knots are exactly those which are $(2,2n+1)-$torus knots, so upon excluding this subfamily, the condition for Lemma \ref{obstruction for alternating knot} is also satisfied. The result is then immediate when substituting Lemma \ref{obstruction for alternating knot} into Theorem \ref{v_3 decomp}. 
\end{proof}

We are interested in the contrapositive of Corollary \ref{obstruction_formula}, which is given in Corollary \ref{2.18} below. 

\begin{cor}
    Let $K$ be a knot as in Corollary \ref{obstruction_formula} and with up to 31 crossings. If $K$ satisfies
    \begin{align*}
        v_3\left(\frac{1}{2} \det(K) + 3g - \frac{5}{2}\right) \neq 7a_2^2 - a_2 - 10a_4
    \end{align*} 
    from Theorem \ref{obstruction_formula}, then $K$ does not admit chirally cosmetic surgeries.
    \label{2.18}
\end{cor}

\section{Python code}
In this section we provide a brief overview of the Python implementation we used for validating the main result by checking the obstruction formula in Corollary \ref{2.18}. For code resources, consult \\
\href{https://github.com/zl830/chiral_cosmetic_surgery_for_pos_2_bridge_knots}{https://github.com/zl830/chiral\_cosmetic\_surgery\_for\_pos\_2\_bridge\_knots}.
\label{section 3}

\subsection{High level overview}
The code can be summarized as follows:
\begin{enumerate}
    \item Input a maximum number of crossings.
    \item Generate every possible positive 2-bridge knot up to said maximum number of crossings. Each knot is represented as a list $[a_1,...,a_m]$ which corresponds to the continued fraction of $p/q$. 
    \item For each presentation of a knot, calculate the Conway polynomial via a recursive skein relation to get $a_2$ and $a_4$.
    \item Calculate the determinant of the knot by substituting $z^2 = -4$ in the Conway polynomial. 
    \item Find the even continued fraction representation to get genus and $v_3$.
    \item Check if the obstruction formula is satisfied. 
    \item Loop through each knot up to maximum number of crossings. 
\end{enumerate}

\subsection{Skein relation recursion}
To compute the Conway polynomial, recall that a positive two-bridge knot can be arranged as in Figure \ref{normal_twist_notation}. We untangle the knot by resolving the crossings from left to right, and use the skein relation from Definition \ref{conway_skein} to compute the polynomial. At each step, we have the following cases to consider: 
\begin{enumerate}
        \item The leftmost braid is $a_i$ where $i$ is odd and $a_i \geq 3$. The resolutions give the links corresponding to $[a_i - 2, a_{i + 1}, \dots, a_{2n + 1}]$ and $[a_i - 1, a_{i + 1}, \dots, a_{2n + 1}]$ for $L_-$ and $L_0$ respectively.
    \begin{figure}[h]
    \begin{center}
        \includegraphics[scale=0.7]{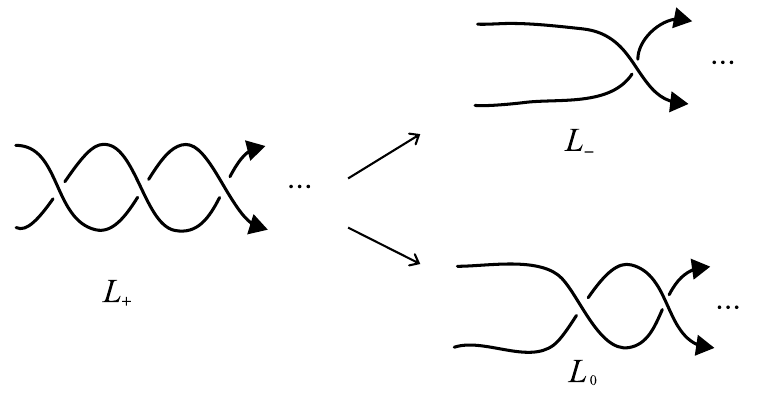}
    \end{center}
    \caption{Case 1}
    \label{odd}
    \end{figure} 
    \item The leftmost braid is $a_i$ where $i$ is odd and $a_i = 2$. The resolutions give the links corresponding to $[0, a_{i + 1}, \dots, a_{2n + 1}]$ and $[1, a_{i + 1}, \dots, a_{2n + 1}]$ for $L_-$ and $L_0$ respectively. In the former case, we can simply untwist the next braid, so we reduce to $[a_{i + 2}, \dots, a_{2n + 1}]$. In the latter case, we have $[a_{i + 1} + 1, a_{i + 2}, \dots, a_{2n + 1}]$. This process is depicted in Figures \ref{odd_2} and \ref{odd_0_and_1}.
    \begin{figure}[h]
    \begin{center}
        \includegraphics[scale=0.7]{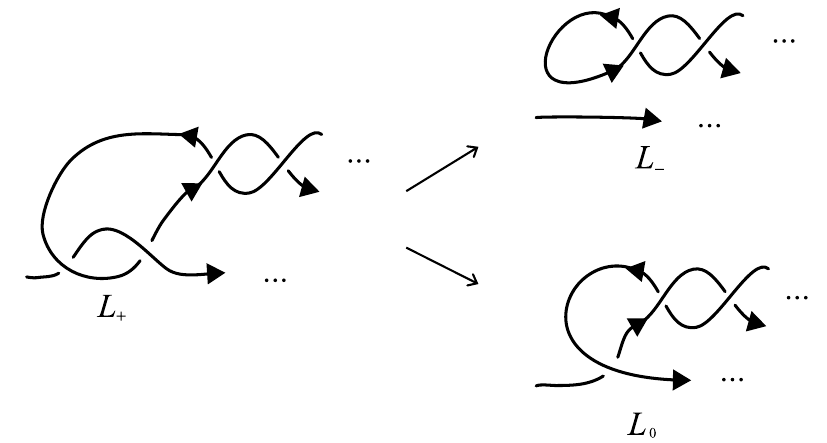}
    \end{center}
    \caption{Case 2}
    \label{odd_2}
    \end{figure} 
    \begin{figure}[h]
    \begin{center}
        \includegraphics[scale=0.7]{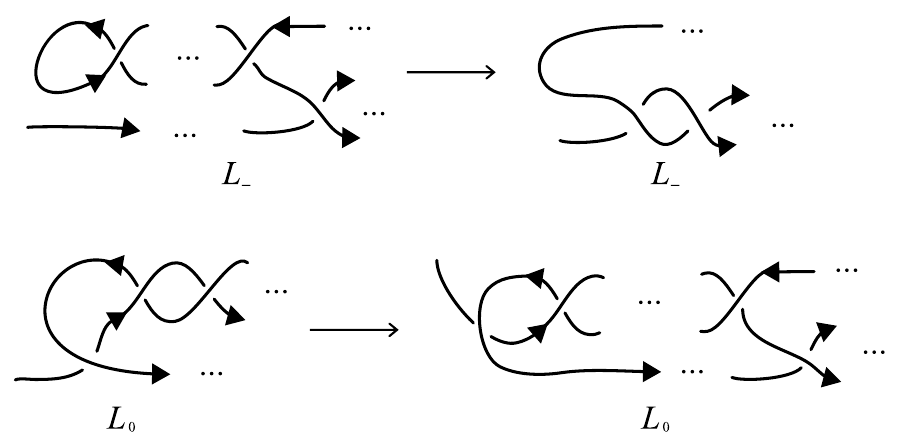}
    \end{center}
    \caption{Case 2 reduced}
    \label{odd_0_and_1}
    \end{figure} 
    \item The leftmost braid is $a_i$ where $i$ is odd and $a_i = 1$ or $a_i = 0$. We use the same reduction in case 2, illustrated in Figure \ref{odd_0_and_1}.
    \item The leftmost braid is $a_i$ where $i$ is even. Notice that from the above process and the fact that $a_i$ is even if $i$ is even for a positive two-bridge knot (Lemma \ref{cond_for_pos}), we have the resolutions in Figure \ref{even}. $L_-$ is the link corresponding to $[a_i - 2, a_{i + 1}, \dots, a_{2n + 1}]$. $L_0$ can be further untwisted and we are left with $[a_{i + 1}, \dots, a_{2n + 1}]$.
    \begin{figure}[h]
    \begin{center}
        \includegraphics[scale=0.7]{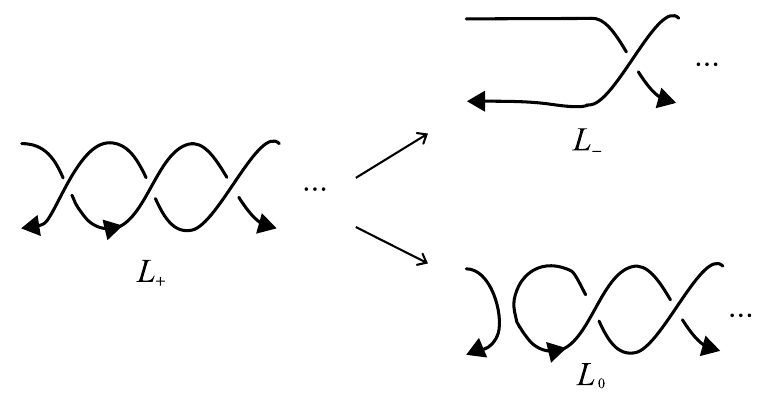}
    \end{center}
    \caption{Case 4}
    \label{even}
    \end{figure}
    \item Resolving the even braids as in case 4 will result in the case where the leftmost braid is $a_i$ where $i$ is even and $a_i = 1$. However, this twist can be simply incorporated into the next braid, resulting in $[a_{i+1} + 1, a_{i + 2}, \dots, a_{2n + 1}]$, as shown in Figure \ref{even_1}
    \begin{figure}[h]
    \begin{center}
        \includegraphics[scale=0.7]{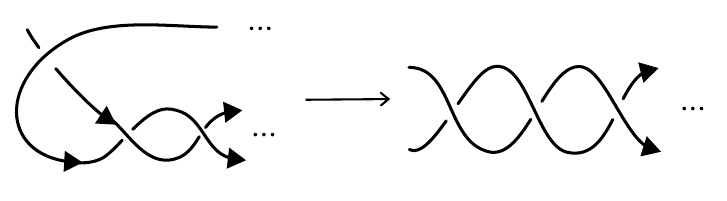}
    \end{center}
    \caption{Case 5}
    \label{even_1}
    \end{figure}
\end{enumerate}
At the end of this process, we are left with either the unknot, or the trivial link with two components. These have Conway polynomial $1$ and $0$ respectively. 

\section{Results}
\subsection{Proof of main result}

We now obtain the main result given in Theorem \ref{main}:

\begin{thm}
    Positive 2-bridge knots with up to 31 crossings do not admit chirally cosmetic surgeries, with the exception of $(2,2n+1)$ torus knots.
\end{thm}

\begin{proof}
    This theorem was verified computationally by running the aforementioned python program for all positive 2-bridge knots up to 31 crossings. No knot resulted in an equality of the obstruction formula. The result thus directly follows from Corollary \ref{2.18}.
\end{proof}

\begin{rmk}
    In theory, the program can verify knots up to any specified number of crossings, but the memory and time complexity scales exponentially with the number of crossings.
\end{rmk}

The 31 crossing threshold was picked because this number allowed us to verify over a million knots, 1,346,268 to be exact. To expedite the recursive process, we used a dynamic database to store the links that were the results of resolutions from the skein relation, for use in more complicated links. Even for such a relatively small number, the program took 100 minutes to complete and used around 8GB of RAM (memory). We have included the values returned (namely the rational number $p/q$, the left-hand side and the right-hand side of obstruction formula \ref{obstruction_formula}, and whether the formula is an equality) in Appendix \ref{outputs}.

\subsection{Future directions and further results}
Although we did not verify the conjecture explicitly for all positive 2-bridge knots, the result of this paper could lead to future work towards the larger result. One possible direction would be proving the general case of inequality in the obstruction formula through an inductive argument. As explained in the skein relation recursion, each positive 2-bridge knot can be simplified into two 2-bridge knots with strictly fewer crossings. Using these resolutions, there may exist a strict inequality involving the sides of the obstruction formula \ref{obstruction_formula} which could be stated in terms of crossing number or continued fraction representation. If such an inequality exists, the Conjecture \ref{chirallyconj} would be proved for all positive 2-bridge knots. 

\begin{rmk}
    In most of the cases our program checked, the left-hand side of the obstruction formula was larger than the right-hand side. The right-hand side is larger when the knot is a $(2, 2n + 1)$ torus knot and when both $p$ and $q$ are relatively small.
    \label{ineq remark}
\end{rmk}
 In light of this observation, we created a plot to compare the complexity of the knot $C(p,q)$ to the quotient obtained from dividing the absolute values of the left-hand side of the obstruction formula by the right-hand side. The $\mathcal{L}^1$ norm was was chosen to measure complexity, and 
\begin{align}
     |p|+|q|
\end{align}
was plotted against 
\begin{align}
    \frac{\left|v_3\left(\frac{1}{2} \det(K) + 3g - \frac{5}{2}\right)\right|}{\left|7a_2^2 - a_2 - 10a_4\right|}
\end{align}
for each positive knot $C(p, q)$ up to 17 crossings. The resulting plot is shown in Figure \ref{quotient}.  Additional plots comparing different measures of complexity against different forms derived from the obstruction formula are included in Appendix \ref{plots}.
\begin{figure}[h]
    \begin{center}
        \includegraphics[scale=0.7]{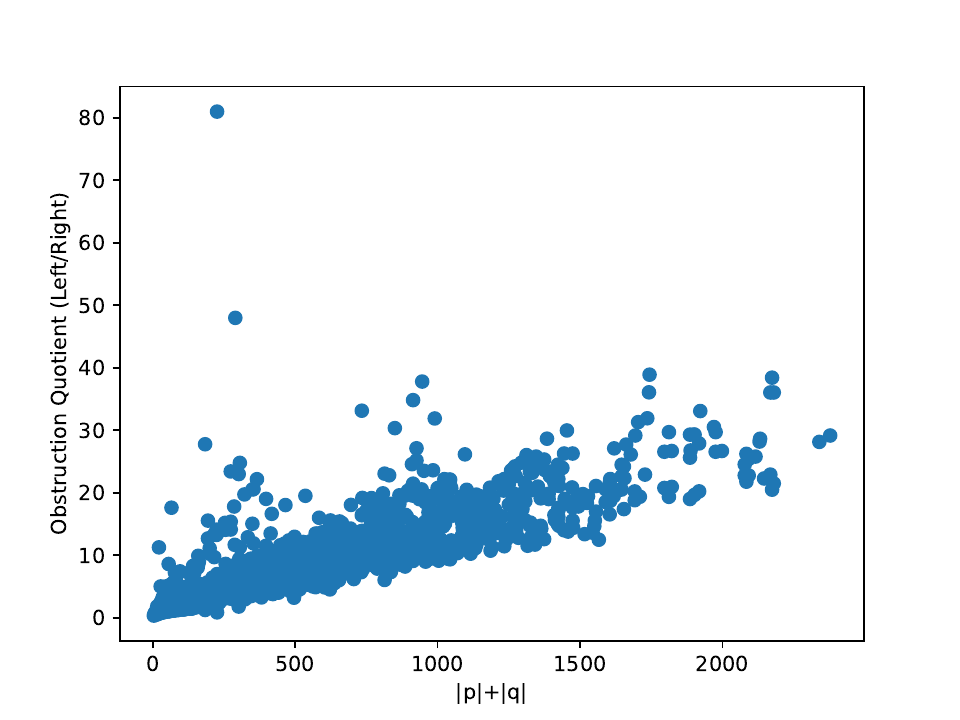}
         \caption{Plot of knot complexity vs. obstruction formula quotient for positive 2-bridge knots up to 17 crossings}
         \label{quotient}
    \end{center}
\end{figure} 

As apparent in Figure \ref{quotient}, the general trend is that the quotient given in (2) increases as the complexity of the knot $C(p,q)$ given by (1) increases, and similar results hold for the plots given in Appendix \ref{plots} as well. For $|p| + |q|$ sufficiently large, we also note that since by Remark \ref{ineq remark} the numerator of (2) is known to always be strictly greater than the denominator of (2) for positive knots $C(p,q)$ up to 17 crossings, the expression given in (2) is always strictly greater than 1 for the plot given in Figure \ref{quotient}. 

If we disregard the cases where $p$ and $q$ are small, these results suggest that for even higher complexity positive 2-bridge knots such as those exceeding 31 crossings, the expression given in (2) will likely continue to be strictly greater than 1. In fact, the trend in the plot from Figure \ref{quotient} suggests that higher complexity knots will result in values of (2) which are actually much greater than 1. This in turn would indicate that the left-hand side of the obstruction formula is always strictly greater than the right-hand side and thus that the obstruction formula never results in equality for any positive 2-bridge knot with the exception of the $(2, 2n + 1)$ torus knots. These ideas lead us to the following conjectures. 

\begin{conj}
    If $K$ is a positive 2-bridge knot, with rational number $p/q$ satisfying $p, q > 20$, then
    \[
        v_3\left(\frac{1}{2} \det(K) + 3g - \frac{5}{2}\right) > 7a_2^2 - a_2 - 10a_4   
    \]
    \label{rhs>lhs}
\end{conj}
If Conjecture \ref{rhs>lhs} is shown to be true, then Conjecture \ref{big_conjecture} below follows immediately from Corollary \ref{2.18}, along with a separate proof for knots of lower complexity (with small $p$ and $q$). 

\begin{conj}
    Positive 2-bridge knots do not admit chirally cosmetic surgeries, with the exception of $(2,2n+1)$ torus knots. 
    \label{big_conjecture}
\end{conj}

Looking at Figure \ref{quotient}, there seems to be some linear patterns of scatter points with different slopes in the cone shaped region. We then tried to separate the linear patterns by their slopes by plotting $|q|$ against the obstruction quotient divided by $|q|$, which is 
\[
s(K) := \frac{\left|v_3\left(\frac{1}{2} \det(K) + 3g - \frac{5}{2}\right)\right|}{\left|7a_2^2 - a_2 - 10a_4\right||q|}
\]

\begin{figure}[h]
    \begin{center}
        \includegraphics[scale=0.7]{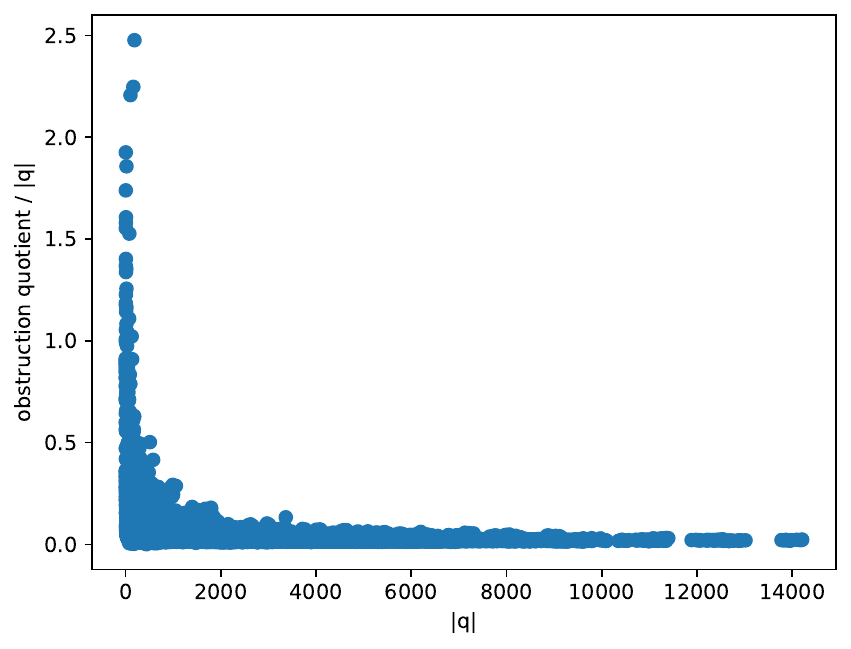}
         \caption{Plot of $|q|$ against obstruction quotient divided by $|q|$, for positive two bridge knots with under 23 crossings}
         \label{slope_vs_q}
    \end{center}
\end{figure} 
\vspace{-10pt}
After thresholding the $|p|$ vs. $|q|$ graph with the value of $s(K)$, we get figure \ref{p_vs_q}.
\begin{figure}[h]
    \begin{center}
        \includegraphics[scale=0.4]{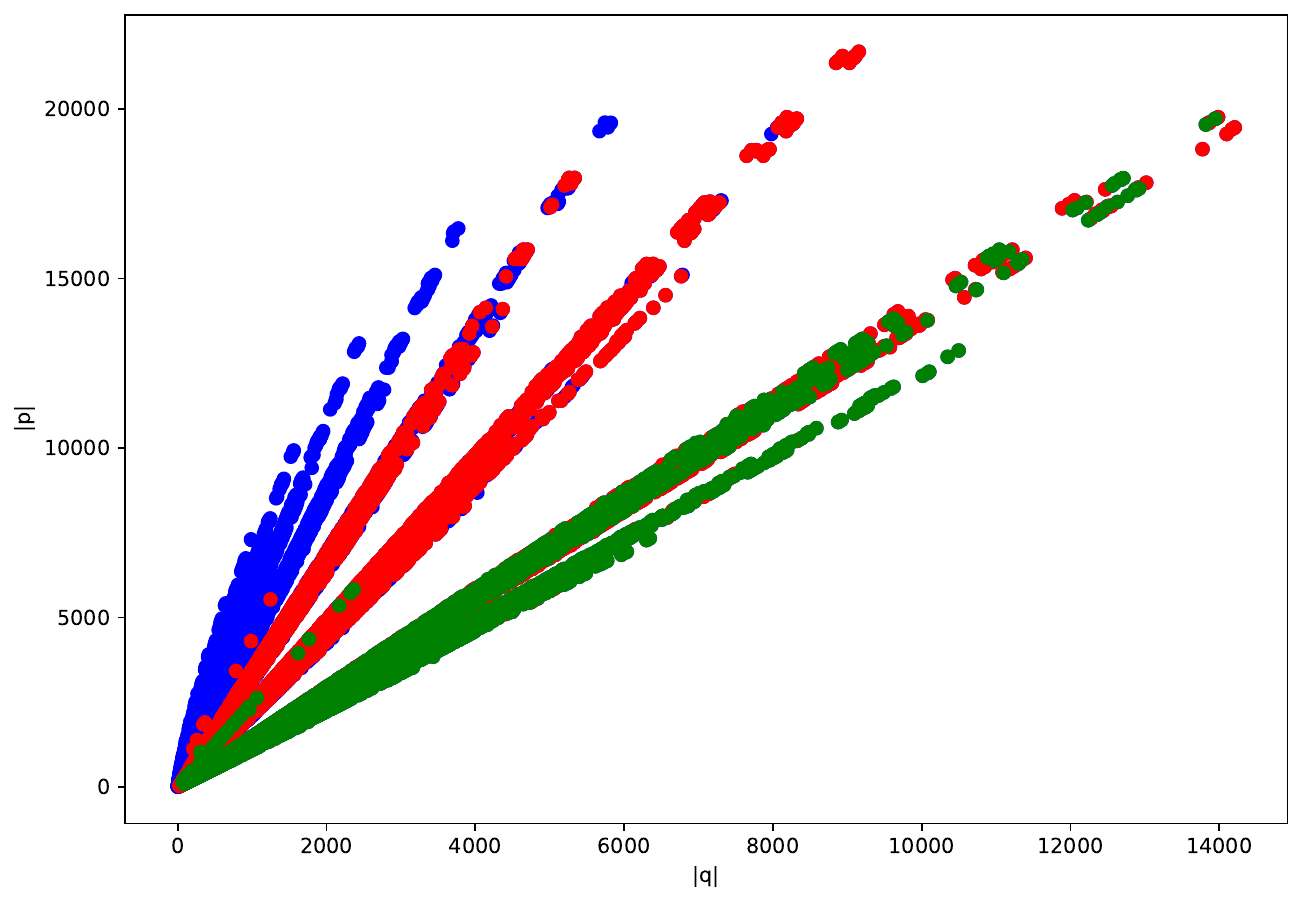}
         \caption{Plot of $|p|$ against $|q|$, where the points labeled red are knots with $s(K) \in (0.02, 0.045)$, and the points labeled green are knots with $s(K) < 0.02$, and points labeled blue are knots with $s(K) > 0.045$}
         \label{p_vs_q}
    \end{center}
\end{figure} 
For positive two bridge knots, remember that $det(K) = p$. Then 
\[
    s(K) = \frac{|v_3(\frac{1}{2}p + 3g - \frac{5}{2})|}{|7a_2^2 - a_2 - 10a_4||q|}.
\]
The general pattern we can the observe from Figure \ref{p_vs_q} is that $s(K)$ increases as $\frac{|p|}{|q|}$ increases.

\section{Acknowledgements}
This work was partially supported by NSF RTG grant No. 2136090. The authors would like to thank Hyunki Min, S\"umeyra Sakall{\i}, and Konstantinos Varvarezos for their guidance and support on this project. The authors would also like to thank Sucharit Sarkar for many helpful discussions and for hosting this REU.

\section{Appendix}
\subsection{Seifert Form}
\label{seifert_form}
In this section, we give an alternative way to calculate the Alexander (and thus Conway) polynomial of a positive 2-bridge knot. 

Recall that we can get the Seifert matrix (or form) $S$ of a knot $K$ by the $n \times n$ matrix of integers $S_{ij} = \text{lk}(x_i, x_j^+)$. Here, the $x_i$ are the generators of $H_1(S^3\setminus K)$, $x_i^+$ is the positive pushoff of $x_i$, and lk is the linking number. This matrix can be written out explicitly for positive 2-bridge knots in the following form.

\begin{thm}
Given a continued fraction representation consistent with Lemma \ref{cond_for_pos}, we can write the Seifert form in the following block diagonal form.  

$$S = \begin{pmatrix}
    A_1 & 0 & \cdots & \cdots & 0 \\
    1 & -1-\frac{a_2}{2} & 0& \cdots & 0\\
    0 & 1 & A_3 & \cdots & \vdots \\
    \vdots & \ddots & \ddots & \ddots & 0\\
    0 & \cdots & 0 & 1 & A_{2n + 1} 
\end{pmatrix}$$

with $a_i - 1 \times a_i - 1$ blocks
\[
    A_i = \begin{pmatrix}
        -1 & 0 & \cdots & & 0 \\
        1 & -1  & \ddots & & 0 \\
        0 & \ddots & \ddots & & \vdots \\
        \vdots & \ddots & \ddots & \ddots & 0 \\
        0 & \cdots & 0 & 1 & -1
    \end{pmatrix}.
\]

\end{thm}
\begin{proof}
    This result can be seen by considering a projection of the type in Figure \ref{normal_twist_notation} and explicitly calculating linking numbers. 
\end{proof}

With the Seifert form for $K$, we can directly get the Alexander polynomial $\Delta_K(t) = \det\left(tS - S^T\right)$. 
\newpage
\subsection{Sample output}
\label{outputs}
Here, we included the sample output from our code when we test 2-bridge positive knots up to 11 crossings.
\begin{multicols}{2}
\begin{figure}[H]
    \begin{center}
        \includegraphics[scale=0.65]{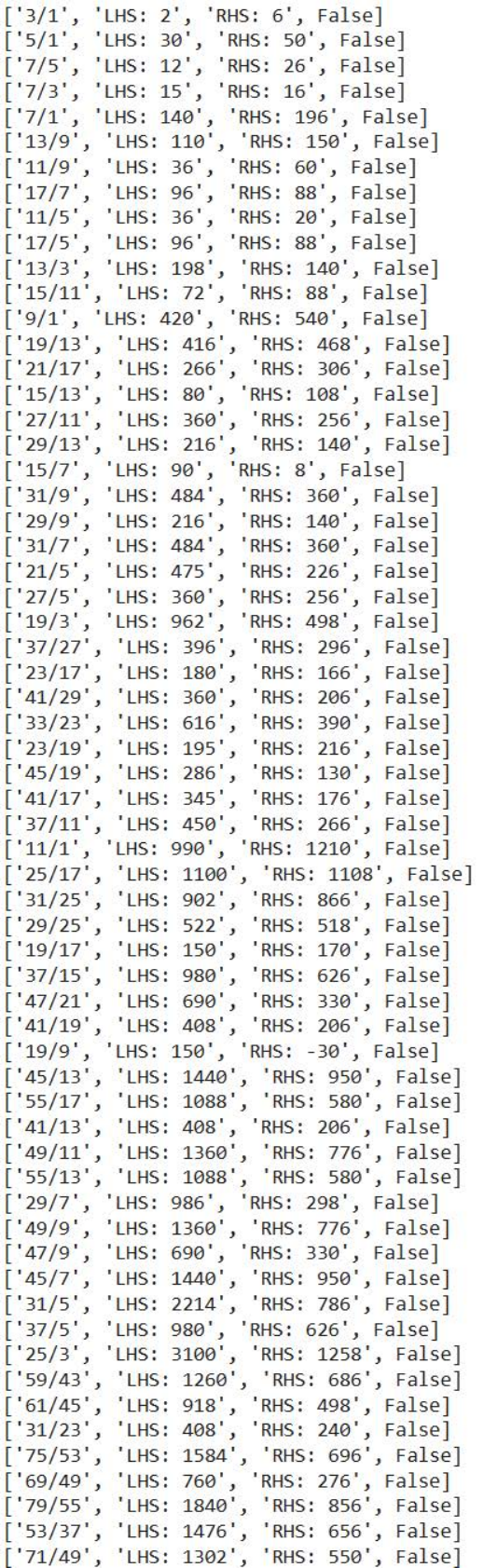}
    \end{center}
\end{figure}
\columnbreak
\begin{figure}[H]
    \begin{center}
        \includegraphics[scale=0.65]{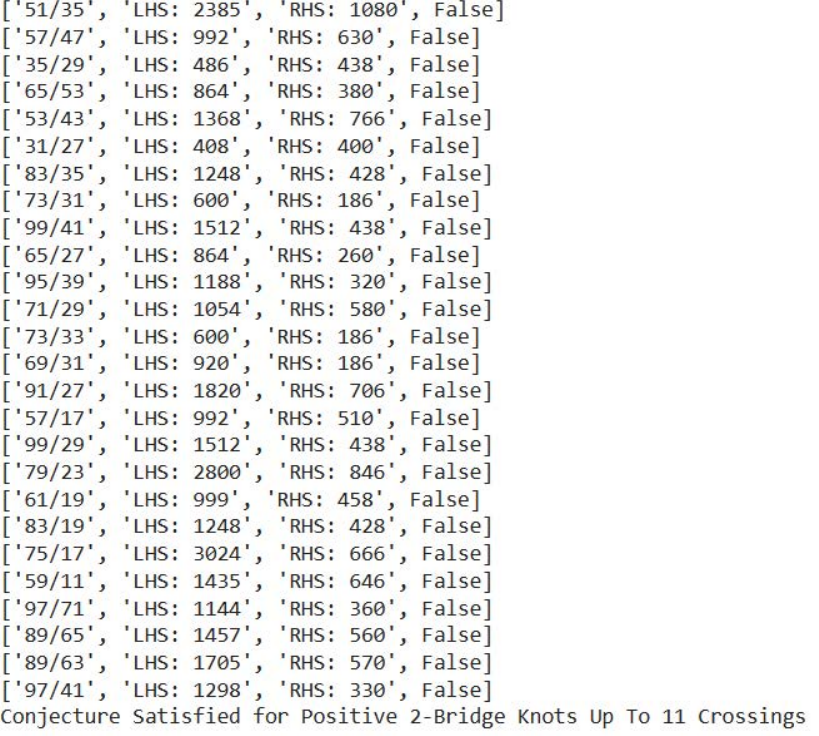}
    \end{center}
\end{figure}
\end{multicols}

\subsection{Additional Plots}
\label{plots}
Here we include some additional plots comparing knot complexity and variations in the obstruction formula. In every measure of complexity, we see a positive correlation between complexity and the magnitude of inequality for the obstruction formula. 

\begin{figure}[h]
    \begin{center}
        \includegraphics[scale=0.65]{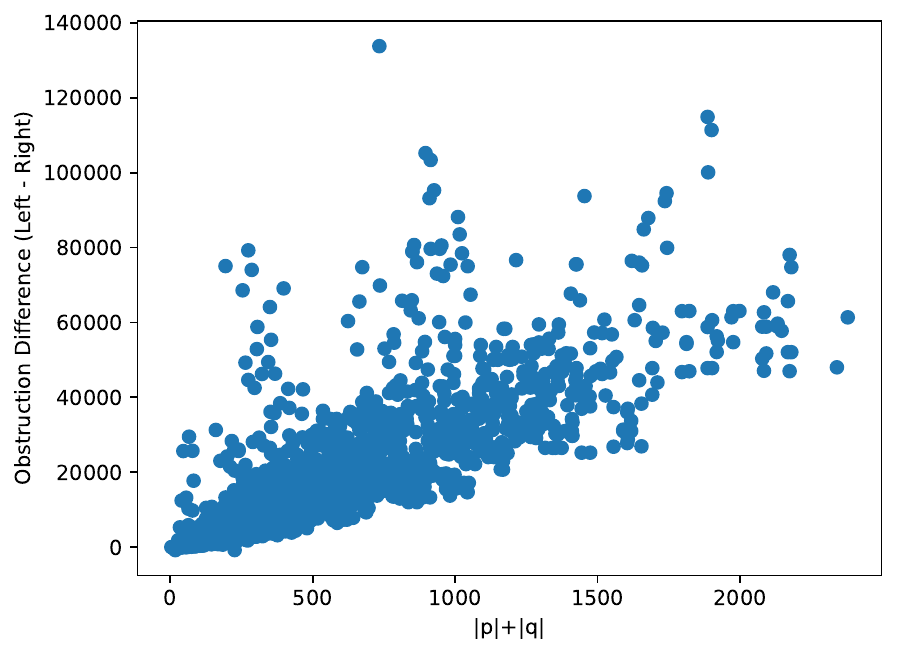}
         \caption{Plot of knot complexity vs. obstruction formula difference for positive 2-bridge knots up to 17 crossings}
    \end{center}
\end{figure} 

\begin{figure}[H]
    \begin{center}
        \includegraphics[scale=0.65]{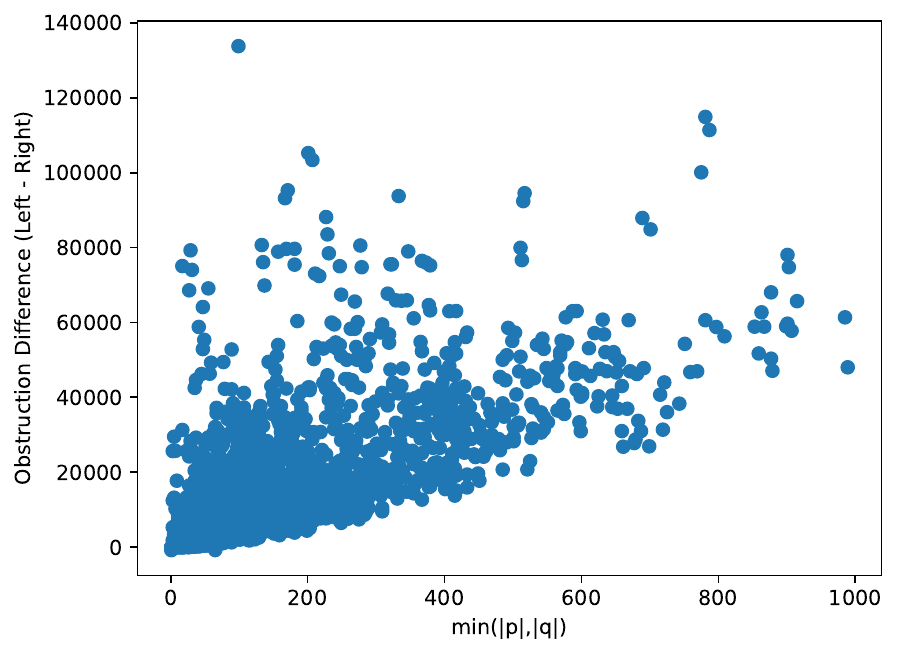}
         \caption{Plot of knot complexity ($\min(|p|,|q|)$ vs. obstruction formula difference for positive 2-bridge knots up to 17 crossings}
    \end{center}
\end{figure} 

\begin{figure}[h]
    \begin{center}
        \includegraphics[scale=0.65]{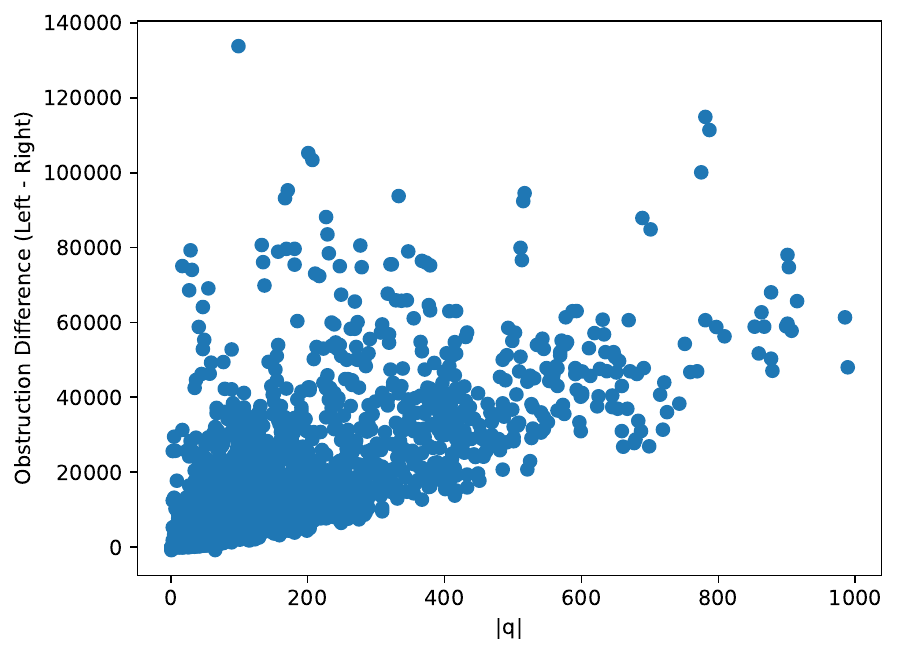}
         \caption{Plot of knot complexity $(|q|)$ vs. obstruction formula difference for positive 2-bridge knots up to 17 crossings}
    \end{center}
\end{figure} 

\begin{figure}[H]
    \begin{center}
        \includegraphics[scale=0.65]{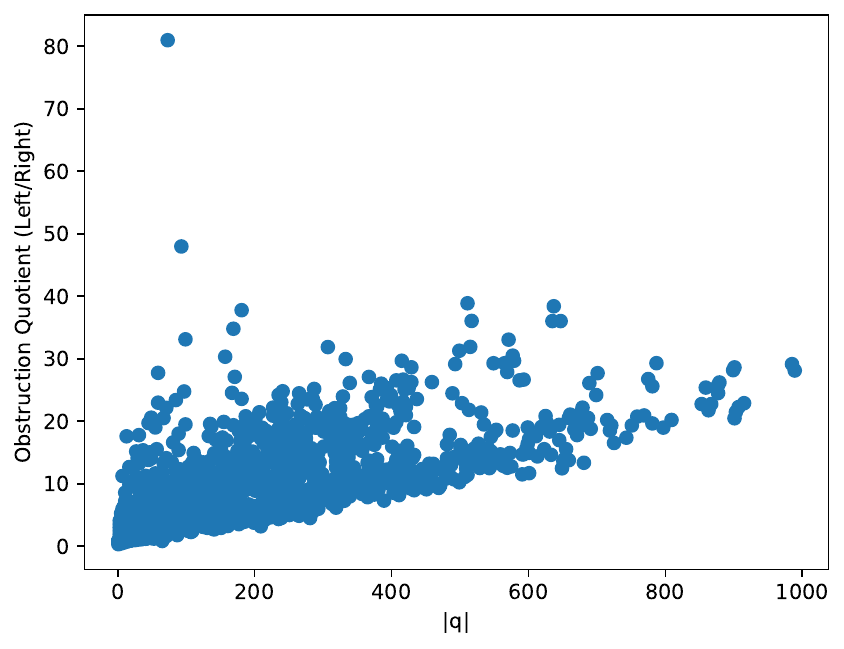}
         \caption{Plot of knot complexity $(|q|)$ vs. obstruction formula quotient for positive 2-bridge knots up to 17 crossings}
    \end{center}
\end{figure} 
\clearpage
\printbibliography
\end{document}